\documentclass[a4paper, twoside,12pt]{article}
\usepackage{fancyhdr}
\usepackage{fnpos}
 \usepackage[english]{babel}        %% adapte le style article aux conventions francophones
\usepackage[T1]{fontenc}          %% permet d'utiliser les caractï؟½res accentuï؟½s
\usepackage{graphicx}             %% permet d'importer des graphiques au format .EPS (postscript)
\usepackage{makeidx}
\usepackage{fancybox}
\usepackage{framed}
\usepackage{fancyhdr}
 \usepackage{pstricks,pst-plot,pstricks-add}
\usepackage[margin=1in]{geometry}
\usepackage{graphicx}
\usepackage{titlesec}
\usepackage{amsmath}
\usepackage{amsfonts}   % math fonts
\usepackage{amssymb}    % extra math symbols
\usepackage{amsthm}
\usepackage{dsfont}
\usepackage{mathtools}
\usepackage{verbatim}   % \begin{comment} multi-line comments \end{comment}
\usepackage{color}
\usepackage{listings}
\usepackage{graphicx}
\usepackage{amsmath}
\usepackage{amsmath,amssymb,color,inputenc,euscript,graphicx,psfrag}

\DeclareMathOperator*{\esssup}{ess\,sup}

\usepackage{amsmath, amsthm, amscd, amsfonts, amssymb, graphicx, color,   mathrsfs}%, showkeys}
\usepackage[english]{babel}
\newtheorem{thm}{Theorem}[section]
\newtheorem{cor}[thm]{Corollary}
\newtheorem{lem}[thm]{Lemma}
\newtheorem{prop}[thm]{Proposition}
\newtheorem{defn}[thm]{Definition}
\newtheorem{rem}[thm]{Remark}
\numberwithin{equation}{section}

\renewcommand{\thefootnote}

\renewcommand\Im{\operatorname{Im}}

 {\markboth{\sl A. Saoudi }{\sl Two-wavelet theory in Weinstein setting}

\newcommand{\X}{\mathcal{X}}
\author {Ahmed Saoudi}

\title{Two-wavelet theory in Weinstein setting}

\date{}
 
\begin{document}
 \maketitle
\begin{center}
     Northern Border University, College of Science, Arar, P.O. Box 1631, Saudi Arabia.\\
   Universit\'{e} de Tunis El Manar, Facult\'{e} des sciences de Tunis, Tunisie.\\
     \textbf{ e-mail:} ahmed.saoudi@ipeim.rnu.tn
\end{center}
  \begin{abstract}
In this paper we introduce the notion of a Weinstein two-wavelet. Then we establish and prove the resolution
of the identity formula for the Weinstein continuous wavelet transform. Next,  we give results on Calder\'on's type reproducing formula in the context of the Weinstein two-wavelet.

 \end{abstract}
  \section{ Introduction}
  The Weinstein operator $\Delta_{W,\alpha}^d$ defined on $\mathbb{R}_{+}^{d+1}=\mathbb{R}^d\times(0, \infty)$, by
\begin{equation*}
\Delta_{W,\alpha}^d=\sum_{j=1}^{d+1}\frac{\partial^2}{\partial x_j^2}+\frac{2\alpha+1}{x_{d+1}}\frac{\partial}{\partial x_{d+1}}=\Delta_d+L_\alpha,\;\alpha>-1/2,
\end{equation*}
where $\Delta_d$ is the Laplacian operator for the $d$ first variables and $L_\alpha$ is the Bessel operator for the last variable defined on $(0,\infty)$ by
$$L_\alpha u=\frac{\partial^2 u}{\partial x_{d+1}^2}+\frac{2\alpha+1}{x_{d+1}}\frac{\partial u}{\partial x_{d+1}}.$$
The Weinstein operator $\Delta_{W,\alpha}^d$ has several applications in pure and applied mathematics, especially in fluid mechanics \cite{brelot1978equation, weinstein1962singular}.

Very recently, many authors have been investigating the behaviour of the Weinstein
transform (\ref{defWeinstein}) with respect to several problems already studied for the classical Fourier transform.
For instance,   Heisenberg-type inequalities \cite{salem2015heisenberg}, Littlewood-Paley g-function \cite{salem2016littlewood},
Shapiro and Hardy–Littlewood–Sobolev type inequalities \cite{salem2020hardy, salem2015shapiro},
 Paley-Wiener theorem \cite{mehrez2017paley}, Uncertainty principles \cite{mejjaoli2011uncertainty, ahmed2018variation, saoudi2019l2}, multiplier Weinstein operator \cite{ahmed2018calder}, wavelet and continuous wavelet transform \cite{gasmi2016inversion, mejjaoli2017new}, Wigner transform and localization operators \cite{saoudi2019weinstein, saoudilocalisation}, and so forth...

In the classical setting, the notion of wavelets was first introduced by Morlet in connection with his study of seismic traces and the mathematical foundations were given by Grossmann and Morlet \cite{grossmann1984decomposition}. Later, Meyer and many other mathematicians recognized many classical results of this theory \cite{koornwinder1993continuous, meyer1992wavelets}. Classical wavelets have wide applications, ranging from signal analysis in geophysics and acoustics to quantum theory and pure mathematics \cite{daubechies1992ten, goupillaud1984cycle, holschneider1995wavelets}.

Recently, the theory of wavelets and continuous wavelet transform has been extended and generalized in the context of differential-differences operators \cite{gasmi2016inversion, mejjaoli2017dunkl, mejjaoli2017new,  mejjaoli2017time}.

Wavelet analysis has attracted attention for its ability to analyze rapidly changing transient signals. Any application using the Fourier like transform can be formulated using wavelets to provide more accurately localized temporal
and frequency information. The reason for the extension from one wavelet to two wavelets comes from the extra
degree of flexibility in signal analysis and imaging when the localization operators are
used as time-varying filters. This paper is an attempt to fill this gap by extending one wavelet to two wavelets in  Weinstein setting.

Using the harmonic analysis associated with the Weinstein operator (generalized translation operators,
generalized convolution, Weinstein transform, ...) and the same idea as for the classical case,
we define and study in this paper the notion of a Weinstein two-wavelet.
For $\varphi$ and $\psi$ be in $L^2_{\alpha}(\mathbb{R}^{d+1}_+)$, the pair $(\varphi,\psi)$ is said  a Weinstein
two-wavelet on $\mathbb{R}^{d+1}_+$ if the following integral
\begin{equation*}
  C_{\varphi,\psi}=\int_{0}^{\infty}\mathcal{F}_{W}(\psi)(a\xi)\overline{\mathcal{F}_{W}(\varphi)}
  (a\xi)\frac{da}{a}
\end{equation*}
is constant for almost all  $\xi\in \mathbb{R}^{d+1}_+$.
We prove for  this Weinstein two-wavelet a Parseval type formula
  \begin{equation*}
    \int_{\X}\Phi^W_\varphi(f)(a,x)\overline{\Phi^W_\psi(g)(a,x)}d\mu_\alpha(a,x)=C_{\varphi,\psi}\int_{\mathbb{R}^{d+1}_+}f(x)\overline{g(x)} d\mu_\alpha(x),
  \end{equation*}
for all $\varphi$ and $\psi$ in $L^2_{\alpha}(\mathbb{R}^{d+1}_+)$ where $\Phi^W_\varphi$ is a Weinstein continuous wavelet transform  on $\mathbb{R}^{d+1}_+$  defined for regular functions $f$ on $\mathbb{R}^{d+1}_+$  in \cite{mejjaoli2017new} by
  \begin{equation*}\label{contwave}
    \forall (a,x)\in\X,\quad \Phi^W_\varphi(f)(a,x)=\int_{\mathbb{R}^{d+1}_+}f(y)\overline{\varphi_{a,x}(y)} d\mu_\alpha(y)=\langle f, \varphi_{a,x}\rangle_{\alpha,2}.
  \end{equation*}
Next, we prove for the pair $(\varphi,\psi)$ an inversion formula for all $f\in L^1_{\alpha}(\mathbb{R}^{d+1}_+)$ (resp. $L^2_{\alpha}(\mathbb{R}^{d+1}_+)$) such that $\mathcal{F}_{W}(f)$  belongs to $f\in L^1_{\alpha}(\mathbb{R}^{d+1}_+)$
  (resp. $L^1_{\alpha}(\mathbb{R}^{d+1}_+)$ $\cap L^\infty_{\alpha}(\mathbb{R}^{d+1}_+)$), of the form
  \begin{equation*}
    f(y)=\frac{1}{C_{\varphi,\psi}}\int_{0}^{\infty}\int_{\mathbb{R}^{d+1}_+}\Phi^W_\varphi(f)(a,x)\psi_{a,x}(y)\mu_\alpha(a,x),
  \end{equation*}
where for each $y\in\mathbb{R}^{d+1}_+ $\, both the inner integral and the outer integral are absolutely
convergent, but eventually not the double integral.
 In the end, we prove for the pair $(\varphi,\psi)$ a Calder\'on's type formulas. For $\varphi$ and $\psi$ be two Weinstein wavelets in $L^2_{\alpha}(\mathbb{R}^{d+1}_+)$ such that  $(\varphi,\psi)$ be a Weinstein two-wavelet, $C_{\varphi,\psi}\neq 0$ and $\mathcal{F}_{W}(\varphi)$, and $\mathcal{F}_{W}(\psi)$ are in $L^\infty_{\alpha}(\mathbb{R}^{d+1}_+)$. Then for all $f$ in $L^2_{\alpha}(\mathbb{R}^{d+1}_+)$, the function
 \begin{equation*}
   f_{\gamma,\delta}(y)=\frac{1}{C_{\varphi,\psi}}\int_{\gamma}^{\delta}\int_{\mathbb{R}^{d+1}_+}
   \Phi^W_\varphi(f)(a,x)\psi_{a,x}(y)d\mu_\alpha(x)\frac{da}{a^{2\alpha+d+3}}, \quad y\in \mathbb{R}^{d+1}_+
 \end{equation*}
  belongs to $L^2_{\alpha}(\mathbb{R}^{d+1}_+)$ and satisfies
 \begin{equation*}\label{limf}
   \lim_{(\gamma,\delta)\to (0,\infty)}\|f_{\gamma,\delta}-f\|_{\alpha,2}=0.
 \end{equation*}

This paper is organized as follows. In Section 2, we recall some properties of harmonic
analysis for the Weinstein operators. In Section 3, we define    the two-wavelet in Weinstein  in setting and prove for it a Parseval type formula and we establish for it an inversion formula. In last Section, we introduce a
Calder\'on's type reproducing formula.
 \section{Preliminaires}

For all $\lambda=(\lambda_1,...,\lambda_{d+1})\in\mathbb{C}^{d+1}$, the system
\begin{equation}
\begin{gathered}
\frac{\partial^2u}{\partial x_{j}^2}(  x)
  =-\lambda_{j} ^2u(x), \quad\text{if } 1\leq j\leq d \\
L_{\alpha}u(  x)  =-\lambda_{d+1}^2u(  x), \\
u(  0)  =1, \quad \frac{\partial u}{\partial
x_{d+1}}(0)=0,\quad \frac{\partial u}{\partial
x_{j}}(0)=-i\lambda_{j}, \quad \text{if } 1\leq j\leq d
\end{gathered}
\end{equation}
 has a unique solution  denoted by $\Lambda_{\alpha}^d(\lambda,.),$ and given by
\begin{equation}\label{wkernel}
\Lambda_{\alpha}^d(\lambda,x)=e^{-i<x^\prime,\lambda^\prime>}j_\alpha(x_{d+1}\lambda_{d+1})
\end{equation}
 where $x=(x^\prime,x_{d+1}),\; x_d'=(x_1,x_2,\cdots,x_d),\; \lambda=(\lambda^\prime,\lambda_{d+1}) ,\; \lambda_d'=(\lambda_1,\lambda_2,\cdots,\lambda_d)$ and $j_\alpha$ is the normalized Bessel function of index $\alpha$ defined by
$$j_\alpha(x)=\Gamma(\alpha+1)\sum_{k=0}^\infty\frac{(-1)^k x^{2k}}{2^k k!\Gamma(\alpha+k+1)}.$$
The function $(\lambda,x)\mapsto\Lambda_{\alpha}^d(\lambda,x)$ is called the Weinstein kernel and has a unique extension to $\mathbb{C}^{d+1}\times\mathbb{C}^{d+1}$, and satisfied the following properties.\\
\begin{itemize}
\item[(i)] For all $(\lambda,x)\in \mathbb{C}^{d+1}\times\mathbb{C}^{d+1}$ we have
\begin{equation*}
\Lambda_{\alpha}^d(\lambda,x)=\Lambda_{\alpha}^d(x,\lambda).
\end{equation*}
\item[(ii)] For all $(\lambda,x)\in \mathbb{C}^{d+1}\times\mathbb{C}^{d+1}$ we have
\begin{equation*}
\Lambda_{\alpha}^d(\lambda,-x)=\Lambda_{\alpha}^d(-\lambda,x).
\end{equation*}
\item[(iii)] For all $(\lambda,x)\in \mathbb{C}^{d+1}\times\mathbb{C}^{d+1}$ we get
\begin{equation*}
\Lambda_{\alpha}^d(\lambda,0)=1.
\end{equation*}
\item[(iv)] For all $\nu\in\mathbb{N}^{d+1},\;x\in\mathbb{R}^{d+1}$ and $\lambda\in\mathbb{C}^{d+1}$ we have
\begin{equation*}\label{klk}
 \left|D_\lambda^\nu\Lambda_{\alpha}^d(\lambda,x)\right|\leq\left\|x\right\|^{\left|\nu\right|}e^{\left\|x\right\|\left\|\Im \lambda\right\|}
\end{equation*}
\end{itemize}
where $D_\lambda^\nu=\partial^\nu/(\partial\lambda_1^{\nu_1}...\partial\lambda_{d+1}^{\nu_{d+1}})$ and $\left|\nu\right|=\nu_1+...+\nu_{d+1}.$ In particular, for all $(\lambda,x)\in \mathbb{R}^{d+1}\times\mathbb{R}^{d+1}$, we have
\begin{equation}\label{normLambda}
\left|\Lambda_{\alpha}^d(\lambda,x)\right|\leq 1.
\end{equation}

In the following we denote by
\begin{itemize}
\item [(i)] $-\lambda=(-\lambda',\lambda_{d+1})$
\item[(ii)] $C_*(\mathbb{R}^{d+1})$, the space of continuous functions on $\mathbb{R}^{d+1},$ even with respect to the last variable.
    \item[(iii)] $S_*(\mathbb{R}^{d+1})$, the space of the $C^\infty$ functions, even with respect to the last variable, and rapidly decreasing together with their derivatives.
  \item[(iv)] $\mathcal{S_*}(\mathbb{R}^{d+1}\times\mathbb{R}^{d+1})$, the Schwartz space of rapidly decreasing functions on $\mathbb{R}^{d+1}\times\mathbb{R}^{d+1}$ even with respect to the last two variables.
  %\item[(v)]  $\mathcal{S'_*}(\mathbb{R}^{d+1}\times\mathbb{R}^{d+1})$, the dual space of $\mathcal{S_*}(\mathbb{R}^{d+1}\times\mathbb{R}^{d+1})$.
  \item[(v)]  $\mathcal{D_*}(\mathbb{R}^{d+1})$, the space of $C^\infty$-functions on $\mathbb{R}^{d+1}$ which are of compact support,even with respect to the last variable.
	\item[(vi)] $L^p_\alpha(\mathbb{R}^{d+1}_+),\;1\leq p\leq \infty,$ the space of measurable functions $f$ on $\mathbb{R}^{d+1}_+$ such that
	$$\left\|f\right\|_{\alpha,p}=\left(\int_{\mathbb{R}^{d+1}_+}\left|f(x)\right|^pd\mu_\alpha(x)\right)^{1/p}<\infty, \;p\in[1,\infty),$$
	$$\left\|f\right\|_{\alpha,\infty}=\textrm{ess}\sup_{x\in\mathbb{R}^{d+1}_+}\left|f(x)\right|<\infty,$$

	%\item $\mathcal{P}_{*,l}^{d+1}$ the set of homogeneous polynomials on $\mathbb{R}^{d+1}$ of degree $l$, even with respect to the last variable.
	%\item $\mathcal{A}_\alpha(\mathbb{R}^{d+1})=\left\{\varphi\in L^1_\alpha(\mathbb{R}^{d+1}_+);\;\mathcal{F}_{W,\alpha}\varphi\in L^1_\alpha(\mathbb{R}^{d+1}_+)\right\}$ the Wiener algebra space.
	%\item $S^d_+=\left\{x\in\mathbb{R}^{d+1}_+;\;\left\|x\right\|=1\right\}.$
where $d\mu_{\alpha}(x)$ is the measure on  $\mathbb{R}_{+}^{d+1}=\mathbb{R}^d\times(0,\infty)$ given by
\begin{equation*}\label{mesure}
	d\mu_\alpha(x)=\frac{x^{2\alpha+1}_{d+1}}{(2\pi)^d2^{2\alpha}\Gamma^2(\alpha+1)}dx.
	\end{equation*}
\end{itemize}

For a radial function $\varphi\in L_{\alpha}^{1}(\mathbb{R}_{+} ^{d+1})$ the function $\tilde{\varphi}$ defined on $\mathbb{R}_+$ such that $\varphi(x)=\tilde{\varphi}(|x|)$, for all
$x\in\mathbb{R}_{+} ^{d+1}$, is integrable with respect to the measure $r^{2\alpha+d+1}dr$, and we have
\begin{equation}\label{radialweinstein}
  \int_{\mathbb{R}_{+}^{d+1}}\varphi(x)d\mu_{\alpha}(x)=a_\alpha\int_{0}^{\infty}
  \tilde{\varphi}(r)r^{2\alpha+d+1}dr,
\end{equation}
where $$a_\alpha=\frac{1}{2^{\alpha+\frac{d}{2}}\Gamma(\alpha+\frac{d}{2}+1)}.$$
The Weinstein transform generalizing the usual Fourier transform, is given for
$\varphi\in L_{\alpha}^{1}(\mathbb{R}_{+} ^{d+1})$ and $\lambda\in\mathbb{R}_{+}^{d+1}$, by

\begin{equation}\label{defWeinstein}
\mathcal{F}_{W}
(\varphi)(\lambda)=\int_{\mathbb{R}_{+}^{d+1}}\varphi(x)\Lambda_{\alpha}^d(x, \lambda
)d\mu_{\alpha}(x),
\end{equation}

We list some known basic properties of the Weinstein transform are as follows. For the proofs, we refer \cite{nahia1996spherical, nahia1996mean}.

\begin{itemize}
	\item[(i)] For all $\varphi\in L^1_\alpha(\mathbb{R}^{d+1}_+)$, the function $\mathcal{F}_{W}(\varphi)$  is continuous on $\mathbb{R}^{d+1}_+$ and we have
	\begin{equation}\label{L1-Linfty}
	\left\|\mathcal{F}_{W}\varphi\right\|_{\alpha,\infty}\leq\left\|\varphi\right\|_{\alpha,1}.
	\end{equation}
	\item[(ii)]   The Weinstein transform is a topological isomorphism from $\mathcal{S}_*(\mathbb{R}^{d+1})$ onto itself. The inverse transform is given by
	\begin{equation}\label{inversionweinstein}
	\mathcal{F}_{W}^{-1}\varphi(\lambda)= \mathcal{F}_{W}\varphi(-\lambda),\;\textrm{for\;all}\;\lambda\in\mathbb{R}^{d+1}_+.
	\end{equation}
	\item[(iii)] For all $f$ in $\mathcal{D_*}(\mathbb{R}^{d+1})$ (resp. $\mathcal{S_*}(\mathbb{R}^{d+1})$),  we have the following relations
\begin{equation}\label{fourierbar}
 \forall\lambda\in\mathbb{R}^{d+1}_+,\quad \mathcal{F}_{W}(\overline{\varphi})(\lambda)= \overline{\mathcal{F}_{W}(\widetilde{\varphi})(\lambda)},
\end{equation}
\begin{equation}\label{fourierbar1}
 \forall\lambda\in\mathbb{R}^{d+1}_+,\quad \mathcal{F}_{W}(\varphi)(\lambda)= \mathcal{F}_{W}(\widetilde{\varphi})(-\lambda),
\end{equation}
where $\widetilde{\varphi}$ is the function defined by
\begin{equation*}
  \forall\lambda\in\mathbb{R}^{d+1}_+,\quad\widetilde{\varphi}(\lambda)=\varphi(-\lambda).
\end{equation*}
	\item[(iv)] Parseval's formula: For all $\varphi, \phi\in \mathcal{S}_*(\mathbb{R}^{d+1})$, we have
	\begin{equation}\label{MM} \int_{\mathbb{R}^{d+1}_+}\varphi(x)\overline{\phi(x)}d\mu_\alpha(x)=\int_{\mathbb{R}^{d+1}_+}\mathcal{F}_{W}
(\varphi)(x)\overline{\mathcal{F}_{W}(\phi)(x)}d\mu_\alpha(x).
	\end{equation}
%\item[(iv)] $L^1-L^\infty$-boundedness:  For all $\varphi\in L^1_\alpha(\mathbb{R}^{d+1}_+)$, $\mathcal{F}_{W}(\varphi)\in L^\infty_\alpha(\mathbb{R}^{d+1}_+)$ and we have
%\begin{equation}\label{Plancherel formula}
	%\left\|\mathcal{F}_{W}\varphi\right\|_{\alpha,\infty}\leq\left\|\varphi\right\|_{\alpha,1}.
	%\end{equation}
\item[(v)] Plancherel's formula: For all $\varphi\in L^2_\alpha(\mathbb{R}^{d+1}_+)$, we have
\begin{equation}\label{Plancherel formula}
	\left\|\mathcal{F}_{W}\varphi\right\|_{\alpha,2}=\left\|\varphi\right\|_{\alpha,2}.
	\end{equation}
\item[(vi)] Plancherel Theorem: The Weinstein transform $\mathcal{F}_{W}$ extends uniquely to an isometric isomorphism on $L^2_\alpha(\mathbb{R}^{d+1}_+).$
\item[(vii)] Inversion formula: Let $\varphi\in L^1_\alpha(\mathbb{R}^{d+1}_+)$ such that $\mathcal{F}_{W}\varphi\in L^1_\alpha(\mathbb{R}^{d+1}_+)$,  then we have
\begin{equation}\label{inv}
\varphi(\lambda)=\int_{\mathbb{R}^{d+1}_+}\mathcal{F}_{W}\varphi(x)\Lambda_{\alpha}^d(-\lambda,x)d\mu_\alpha(x),\;\textrm{a.e. }\lambda\in\mathbb{R}^{d+1}_+.
\end{equation}
\end{itemize}

Using relations (\ref{L1-Linfty}) and (\ref{Plancherel formula}) with Marcinkiewicz's interpolation theorem \cite{zbMATH03367521} we deduce that for every $\varphi\in L^p_\alpha(\mathbb{R}^{d+1}_+)$ for all $1\leq p\leq 2$, the function $\mathcal{F}_{W}(\varphi)\in L^q_\alpha(\mathbb{R}^{d+1}_+), q=p/(p-1),$ and
\begin{equation}\label{Lp-Lq}
	\left\|\mathcal{F}_{W}\varphi\right\|_{\alpha,q}\leq\left\|\varphi\right\|_{\alpha,p}.
	\end{equation}

\begin{defn} The translation operator $\tau^\alpha_x,\;x\in\mathbb{R}^{d+1}_+$ associated with the Weinstein operator $\Delta_{W,\alpha}^d$, is defined for a continuous function $\varphi$ on $\mathbb{R}^{d+1}_+$, which is even with respect to the last variable and for all $y\in\mathbb{R}^{d+1}_+$ by
$$\tau^\alpha_x\varphi(y)=C_\alpha\int_0^\pi\varphi\left(x^\prime+y\prime,\sqrt{x^2_{d+1}+y^2_{d+1}+2x_{d+1}y_{d+1}
\cos\theta}\right)\left(\sin\theta\right)^{2\alpha}d\theta,$$
with $$C_\alpha=\frac{\Gamma(\alpha+1)}{\sqrt{\pi}\Gamma(\alpha+1/2)}.$$
\end{defn}
By using the Weinstein kernel, we can also define a generalized translation, for
a function $\varphi\in\mathcal{S}_*(\mathbb{R}^{d+1})$ and $y\in\mathbb{R}^{d+1}_+$ the generalized translation $\tau^\alpha_x\varphi$ is defined by the following relation
\begin{equation}\label{MMM}
\mathcal{F}_{W}(\tau^\alpha_x\varphi)(y)=\Lambda^d_\alpha(x,y)\mathcal{F}_{W}(\varphi)(y).
\end{equation}

In the following proposition, we give some properties of the Weinstein
translation operator:

\begin{prop} The translation operator $\tau^\alpha_x,\;x\in\mathbb{R}^{d+1}_+$ satisfies the following properties.\\
i). For $\varphi\in\mathbb{C}_*(\mathbb{R}^{d+1})$, we have for all $x,y\in\mathbb{R}^{d+1}_+$
\begin{equation}\label{symtrictrans}
\tau^\alpha_x\varphi(y)=\tau^\alpha_y\varphi(x)\;\textrm{and}\;\tau^\alpha_0\varphi=\varphi.
\end{equation}
ii). Let $\varphi\in L^p_\alpha(\mathbb{R}^{d+1}_+),\;1\leq p\leq \infty$ and $x\in\mathbb{R}^{d+1}_+$. Then  $\tau^\alpha_x\varphi$ belongs to $L^p_\alpha(\mathbb{R}^{d+1}_+)$ and we have
\begin{equation}\label{ineqtransl}
\left\| \tau^\alpha_x\varphi\right\|_{\alpha,p}\leq \left\|\varphi\right\|_{\alpha,p}.
\end{equation}
\end{prop}
\begin{prop}
   Let $\varphi\in L^1_{\alpha}(\mathbb{R}^{d+1}_+)$. Then for all $x\in \mathbb{R}^{d+1}_+$,
    \begin{equation}\label{integraltransrad}
      \int_{\mathbb{R}^{d+1}_+}\tau^\alpha_x\varphi(y)d\mu_\alpha(y)= \int_{\mathbb{R}^{d+1}_+}\varphi(y)  d\mu_\alpha(y).
    \end{equation}
  \end{prop}
  \begin{proof}
    The result comes from combination identities (\ref{inv}) and (\ref{MMM}).
  \end{proof}
  By using the generalized translation, we define the generalized convolution product $\varphi*\psi$ of the functions $\varphi,\;\psi\in L^1_\alpha(\mathbb{R}^{d+1}_+)$ as follows
\begin{equation}\label{defconvolution}
\varphi*\psi(x)=\int_{\mathbb{R}^{d+1}_+}\tau^\alpha_x\varphi(-y)\psi(y)d\mu_\alpha(y).
\end{equation}
\\
This convolution is commutative and associative, and it satisfies the following properties.

\begin{prop}\label{propconvol}
i) For all $\varphi,\psi\in L^1_\alpha(\mathbb{R}^{d+1}_+),$\;(resp. $\varphi,\psi\in \mathcal{S}_*(\mathbb{R}^{d+1})$), then $\varphi*\psi\in L^1_\alpha(\mathbb{R}^{d+1}_+),$\;(resp. $\varphi*\psi\in \mathcal{S}_*(\mathbb{R}^{d+1})$) and we have
\begin{equation}
\mathcal{F}_{W}(\varphi*\psi)=\mathcal{F}_{W}(\varphi)\mathcal{F}_{W}(\psi).
\end{equation}
ii) Let $p, q, r\in [1,\infty],$ such that $\frac{1}{p}+\frac{1}{q}-\frac{1}{r}=1.$ Then for all $\varphi\in L^p_\alpha(\mathbb{R}^{d+1}_+)$ and  $\psi\in L^q_\alpha(\mathbb{R}^{d+1}_+)$ the function $\varphi*\psi$ belongs to  $L^r_\alpha(\mathbb{R}^{d+1}_+)$ and we have
\begin{equation}\label{invconvol2}
\left\|\varphi*\psi\right\|_{\alpha,r}\leq\left\|\varphi\right\|_{\alpha,p}\left\|\psi\right\|_{\alpha,q}.
\end{equation}
iii)  Let $\varphi,\psi\in L^2_\alpha(\mathbb{R}^{d+1}_+)$. Then
\begin{equation}
\varphi*\psi=\mathcal{F}_{W}^{-1}\left(\mathcal{F}_{W}(\varphi)\mathcal{F}_{W}(\psi)\right).
\end{equation}
iv) Let $\varphi,\psi\in L^2_\alpha(\mathbb{R}^{d+1}_+)$. Then $\varphi*\psi$ belongs to $L^2_\alpha(\mathbb{R}^{d+1}_+)$ if and only if $\mathcal{F}_{W}(\varphi)\mathcal{F}_{W}(\psi)$ belongs to $L^2_\alpha(\mathbb{R}^{d+1}_+)$ and we have
\begin{equation}
\mathcal{F}_{W}(\varphi*\psi)=\mathcal{F}_{W}(\varphi)\mathcal{F}_{W}(\psi).
\end{equation}
v)  Let $\varphi,\psi\in L^2_\alpha(\mathbb{R}^{d+1}_+)$. Then
\begin{equation}
\|\varphi*\psi\|_{\alpha,2}=\|\mathcal{F}_{W}(\varphi)\mathcal{F}_{W}(\psi)\|_{\alpha,2},
\end{equation}
where both sides are finite or infinite.
\end{prop}
\section{Weinstein two-wavelet theory}

In the following, we denote by \\
$\X=\left\{(a,x): x\in \mathbb{R}^{d+1}_+ \;\text{and}\; a>0\right\}$.\\
$L^p_{\alpha}(\X),\; p\in [1,\infty]$ the space of measurable functions $\varphi$ on $\X$ such that
\begin{eqnarray*}
  \|\varphi\|_{L^p_{\alpha}(\X)} &:=& \left(\int_{\X}|\varphi(a,x)|^p  d\mu_\alpha(a,x)\right)^\frac{1}{p}<\infty,\quad 1\leq p<\infty, \\
  \|\varphi\|_{L^\infty_{\alpha}(\X)}&=& \esssup_{(a,x)\in\X}|\varphi(a,x)|<\infty,
\end{eqnarray*}
where the measure $\mu_\alpha(a,x)$ is defined on $\X$ by
$$d\mu_\alpha(a,x)=\frac{d\mu_\alpha(x)da}{a^{2\alpha+d+3}}$$
\begin{defn}\cite{gasmi2016inversion}
 A classical wavelet  on $\mathbb{R}^{d+1}_+$ is a measurable function $\varphi$ on $\mathbb{R}^{d+1}_+$
satisfying for almost all $\xi\in \mathbb{R}^{d+1}_+$, the condition
\begin{equation}\label{defwave}
  0<C_\varphi=\int_{0}^{\infty}|\mathcal{F}_{W}(\varphi)(a\xi)|^2\frac{da}{a}<\infty.
\end{equation}
\end{defn}
We extend the notion of the wavelet to the  two-wavelet in Weinstein setting as follows
\begin{defn}
  Let $\varphi$ and $\psi$ be in $L^2_{\alpha}(\mathbb{R}^{d+1}_+)$. We say that the pair $(\varphi,\psi)$ is a Weinstein
two-wavelet on $\mathbb{R}^{d+1}_+$ if the following integral
\begin{equation}\label{deftwowave}
  C_{\varphi,\psi}=\int_{0}^{\infty}\mathcal{F}_{W}(\psi)(a\xi)\overline{\mathcal{F}_{W}(\varphi)}
  (a\xi)\frac{da}{a}
\end{equation}
is constant for almost all  $\xi\in \mathbb{R}^{d+1}_+$ and we call the number $ C_{\varphi,\psi}$ the Weinstein two-wavelet constant associated to the functions $\varphi$ and $\psi$.
\end{defn}
It is to highlight that if $u$ is a Weinstein  wavelet then the pair $(\varphi,\psi)$ is a Weinstein two-wavelet,
and $C_{\varphi,\psi}$ coincides with   $ C_{\varphi}$.

Let $a>0$ and $\varphi$ be a measurable function. We consider the function $\varphi_a$ defined by
\begin{equation}\label{fia}
  \forall x\in \mathbb{R}^{d+1}_+, \quad \varphi_a(x)=\frac{1}{a^{2\alpha+d+2}}\varphi\left(\frac{x}{a}\right).
\end{equation}
\begin{prop}
  \begin{enumerate}
    \item Let $a>0$ and $\varphi\in L^p_{\alpha}(\mathbb{R}^{d+1}_+),\;p\in[1,\infty]$. The function $\varphi_a$ belongs to $L^p_{\alpha}(\mathbb{R}^{d+1}_+)$ and we have
        \begin{equation}\label{normLpfia}
           \|\varphi_a\|_{\alpha,p}=a^{(2\alpha+d+2)(\frac{1}{p}-1)} \|\varphi\|_{\alpha,p}.
        \end{equation}
    \item Let $a>0$ and $\varphi\in L^1_{\alpha}(\mathbb{R}^{d+1}_+)\cup L^2_{\alpha}(\mathbb{R}^{d+1}_+)$. Then, we have
        \begin{equation}\label{fourierfia}
          \mathcal{F}_{W}(\varphi_a)(\xi)=\mathcal{F}_{W}(\varphi)(a\xi),\quad\xi\in \mathbb{R}^{d+1}_+.
        \end{equation}
  \end{enumerate}
\end{prop}
For $a>0$ and $\varphi\in L^2_{\alpha}(\mathbb{R}^{d+1}_+)$, we consider the family $\varphi_{a,x},\; x\in
\mathbb{R}^{d+1}_+$ of Weinstein wavelets on $\mathbb{R}^{d+1}_+$ in $L^2_{\alpha}(\mathbb{R}^{d+1}_+)$ defined by
\begin{equation}\label{deffiax}
  \forall y\in \mathbb{R}^{d+1}_+,\quad \varphi_{a,x}=a^{\alpha+1+\frac{d}{2}}\tau^\alpha_x\varphi_a(y).
\end{equation}
\begin{rem}
  \begin{enumerate}
    \item   Let $\varphi$ be a function in $L^2_{\alpha}(\mathbb{R}^{d+1}_+)$, then we have
    \begin{equation}\label{Nfiax2}
      \forall (a,x)\in\X, \quad \|\varphi_{a,x}\|_{\alpha,2}\leq  \|\varphi\|_{\alpha,2}.
    \end{equation}
    \item  Let $p\in [1,\infty]$ and $\varphi$ be a function in $L^p_{\alpha}(\mathbb{R}^{d+1}_+)$, then we have
    \begin{equation}\label{Nfiaxp}
      \forall (a,x)\in\X, \quad \|\varphi_{a,x}\|_{\alpha,p}\leq a^{(2\alpha+d+2)(\frac{1}{p}-\frac{1}{2})} \|\varphi\|_{\alpha,p}.
    \end{equation}
  \end{enumerate}
\end{rem}
\begin{defn}\cite{mejjaoli2017new}
  Let $\varphi$ be a Weinstein wavelet on $\mathbb{R}^{d+1}_+$ in $L^2_{\alpha}(\mathbb{R}^{d+1}_+)$. The Weinstein continuous wavelet transform $\Phi^W_\varphi$ on $\mathbb{R}^{d+1}_+$ is defined for regular functions $f$ on $\mathbb{R}^{d+1}_+$ by
  \begin{equation}\label{contwave}
    \forall (a,x)\in\X,\quad \Phi^W_\varphi(f)(a,x)=\int_{\mathbb{R}^{d+1}_+}f(y)\overline{\varphi_{a,x}(y)} d\mu_\alpha(y)=\langle f, \varphi_{a,x}\rangle_{\alpha,2}.
  \end{equation}
\end{defn}
This transform can also be written in the form
\begin{equation}\label{recontwave}
  \Phi^W_\varphi(f)(a,x)=a^{\alpha+1+\frac{d}{2}}\check{f}*\overline{\varphi_a}(x).
\end{equation}
\begin{rem}
  \begin{enumerate}
    \item Let $\varphi$ be a function in $L^p_{\alpha}(\mathbb{R}^{d+1}_+)$, and let $f$ be a function in $L^q_{\alpha}(\mathbb{R}^{d+1}_+)$, with $p\in [1,\infty]$, we define the Weinstein continuous wavelet transform $\Phi^W_\varphi(f)$ by the relation (\ref{recontwave}).
    \item Let $\varphi$ be a Weinstein wavelet on $\mathbb{R}^{d+1}_+$ in $L^2_{\alpha}(\mathbb{R}^{d+1}_+)$. Then from the relations (\ref{Nfiax2}) and (\ref{contwave}), we have for all $f\in L^2_{\alpha}(\mathbb{R}^{d+1}_+)$
        \begin{equation}\label{Nwcont2}
           \|\Phi^W_\varphi(f)\|_{\alpha,\infty}\leq \|f\|_{\alpha,2}\|\varphi\|_{\alpha,2}.
        \end{equation}
    \item Let $\varphi$ be a function in $L^p_{\alpha}(\mathbb{R}^{d+1}_+)$, with $p\in [1,\infty]$, then from the inequality (\ref{invconvol2}) and the identity (\ref{recontwave}),  we have for all $f\in L^q_{\alpha}(\mathbb{R}^{d+1}_+)$
        \begin{equation}\label{Nwcontp}
           \|\Phi^W_\varphi(f)\|_{\alpha,\infty}\leq \|f\|_{\alpha,q}\|\varphi\|_{\alpha,p}.
        \end{equation}
  \end{enumerate}
\end{rem}
The following Theorem generalizes the Parseval's formula for the continuous
Weinstein wavelet transform proved by Mejjaoli \cite{mejjaoli2017new}.
\begin{thm}\label{Parseval's formula2wave}
  Let $(\varphi,\psi)$ be a Weinstein two-wavelet. Then for all $f$ and $g$ in $L^2_{\alpha}(\mathbb{R}^{d+1}_+)$, we have the following  Parseval type formula
  \begin{equation}\label{ParsevalPsi}
    \int_{\X}\Phi^W_\varphi(f)(a,x)\overline{\Phi^W_\psi(g)(a,x)}d\mu_\alpha(a,x)=C_{\varphi,\psi}\int_{\mathbb{R}^{d+1}_+}f(x)\overline{g(x)} d\mu_\alpha(x),
  \end{equation}
  where $ C_{\varphi,\psi}$ is the Weinstein two-wavelet constant associated to the functions $\varphi$ and $\psi$ given by the identity (\ref{deftwowave}).
\end{thm}
\begin{proof}
  According to Fubini's Theorem, the relation (\ref{recontwave}) and Parseval's formula for the Weinstein transform (\ref{MM}), we obtain
  \begin{eqnarray*}
 &&\!\!\!\!\!\!\!\!\!\!\!\!\!\!\!\!   \int_{\X}\Phi^W_\varphi(f)(a,x)\overline{\Phi^W_\psi(g)(a,x)}d\mu_\alpha(a,x)\\ &=&   \int_{0}^{\infty}\int_{\mathbb{R}^{d+1}_+}\check{f}*\overline{\varphi_a}(x)\overline{\check{g}*\overline{\psi_a}(x)}
 a^{2\alpha+d+2}d\mu_\alpha(a,x)\\ &=&
  \int_{0}^{\infty}\int_{\mathbb{R}^{d+1}_+}\check{f}*\overline{\varphi_a}(x)\overline{\check{g}*\overline{\psi_a}(x)}
 d\mu_\alpha(x)\frac{da}{a}\\ &=&
  \int_{0}^{\infty}\int_{\mathbb{R}^{d+1}_+}\mathcal{F}_{W}(\check{f})(\xi)\overline{\mathcal{F}_{W}(\check{g})(\xi)}
  \mathcal{F}_{W}(\overline{\varphi_a})(\xi)\overline{\mathcal{F}_{W}(\overline{\psi_a})(\xi)}d\mu_\alpha(\xi)\frac{da}{a}\\ &=&
   \int_{\mathbb{R}^{d+1}_+}\mathcal{F}_{W}(\check{f})(\xi)\overline{\mathcal{F}_{W}(\check{g})(\xi)}
  \left(\int_{0}^{\infty}\mathcal{F}_{W}(\overline{\varphi})(a\xi)\mathcal{F}_{W}(\psi)(-a\xi)\frac{da}{a}\right) d\mu_\alpha(\xi).
   \end{eqnarray*}
   Moreover, using the relations (\ref{fourierbar}) and (\ref{fourierbar1}), we conclude that
   \begin{eqnarray*}
 &&\!\!\!\!\!\!\!\!\!\!\!\!\!\!\!\!   \int_{\X}\Phi^W_\varphi(f)(a,x)\overline{\Phi^W_\psi(g)(a,x)}d\mu_\alpha(a,x)\\ &=&
  \int_{\mathbb{R}^{d+1}_+}\mathcal{F}_{W}(f)(\xi)\overline{\mathcal{F}_{W}(g)(\xi)}
  \left(\int_{0}^{\infty}\overline{\mathcal{F}_{W}(\varphi)(a\xi)}\mathcal{F}_{W}(\psi)(a\xi)\frac{da}{a}\right) d\mu_\alpha(\xi)\\ &=&  C_{\varphi,\psi}\int_{\mathbb{R}^{d+1}_+}\mathcal{F}_{W}(f)(\xi)\overline{\mathcal{F}_{W}(g)(\xi)}d\mu_\alpha(\xi).
 \end{eqnarray*}
 Finally, we get the desired result using the  Parseval's formula for the Weinstein transform (\ref{MM}).
\end{proof}
In the particular case of the previous theorem when $\varphi=\psi$  and $f=g$, we obtain the following Plancherel’s formula  for the Weinstein continuous wavelet transform provided in \cite{mejjaoli2017new}
\begin{equation}\label{ParsevalPsiu=vf=g}
    \int_{\X}\left|\Phi^W_\varphi(f)(a,x)\right|^2d\mu_\alpha(a,x)=C_{\varphi}\int_{\mathbb{R}^{d+1}_+}\left|f(x)\right|^2 d\mu_\alpha(x),
  \end{equation}
  where
  \begin{equation}\label{deftwowaveu=v}
 C_{\varphi}=C_{\varphi,\varphi}=\int_{0}^{\infty}\left|\mathcal{F}_{W}(\varphi)(a\xi)\right|^2\frac{da}{a}.
\end{equation}
From the Parseval type formula in Theorem \ref{Parseval's formula2wave}, we deduce the following orthogonality result.
\begin{cor}
  Let $(\varphi,\psi)$ be a Weinstein two-wavelet. Then we have the following assertion:
 If the Weinstein two-wavelet constant  $C_{\varphi,\varphi}=0$, then $\Phi^W_\varphi\left(L^2_{\alpha}(\mathbb{R}^{d+1}_+)\right)$ and $\Phi^W_\psi\left(L^2_{\alpha}(\mathbb{R}^{d+1}_+)\right)$ are orthogonal.
\end{cor}
\begin{thm}(Inversion formula) Let $(\varphi,\psi)$ be a Weinstein two-wavelet. For all $f\in L^1_{\alpha}(\mathbb{R}^{d+1}_+)$ (resp. $L^2_{\alpha}(\mathbb{R}^{d+1}_+)$) such that $\mathcal{F}_{W}(f)$  belongs to $f\in L^1_{\alpha}(\mathbb{R}^{d+1}_+)$
  (resp. $L^1_{\alpha}(\mathbb{R}^{d+1}_+)$ $\cap L^\infty_{\alpha}(\mathbb{R}^{d+1}_+)$), we have
  \begin{equation}\label{inversion2wave}
    f(y)=\frac{1}{C_{\varphi,\psi}}\int_{0}^{\infty}\int_{\mathbb{R}^{d+1}_+}\Phi^W_\varphi(f)(a,x)\psi_{a,x}(y)\mu_\alpha(a,x),\quad a.e.
  \end{equation}
where for each $y\in\mathbb{R}^{d+1}_+ $\, both the inner integral and the outer integral are absolutely
convergent, but eventually not the double integral.
\end{thm}
\begin{proof}
  Using similar ideas as in the proof in \cite[Theorem 6.III.3]{trimeche2019generalized}, we obtain the  relation (\ref{inversion2wave}).
\end{proof}
\section{Calder\'on's reproducing formulas}
The main result of this section is to establish a Calder\'on's type formulas for the Weinstein two-wavelet transform under the following assumptions:
\begin{itemize}
  \item ($A_1$) Let $\varphi$ and $\psi$ be two Weinstein wavelets in $L^2_{\alpha}(\mathbb{R}^{d+1}_+)$ such that  $(\varphi,\psi)$ be a Weinstein two-wavelet, and $\mathcal{F}_{W}(\varphi)$ and $\mathcal{F}_{W}(\psi)$ are in $L^\infty_{\alpha}(\mathbb{R}^{d+1}_+)$.
  \item ($A_2$) The Weinstein two-wavelet constant  $C_{\varphi,\varphi}\neq 0$.
\end{itemize}

\begin{thm}\label{mainthm} (Calder\'on's type formulas) Let $\varphi$ and $\psi$ be two-Weinstein wavelets  satisfying the assumptions ($A_1$) and ($A_2$) and $0<\gamma<\delta<\infty$. Then for all $f$ in $L^2_{\alpha}(\mathbb{R}^{d+1}_+)$, the function
  \begin{equation}\label{fgamadelta}
   f_{\gamma,\delta}(y)=\frac{1}{C_{\varphi,\psi}}\int_{\gamma}^{\delta}\int_{\mathbb{R}^{d+1}_+}
   \Phi^W_\varphi(f)(a,x)\psi_{a,x}(y)d\mu_\alpha(x)\frac{da}{a^{2\alpha+d+3}}, \quad y\in \mathbb{R}^{d+1}_+
 \end{equation}
  belongs to $L^2_{\alpha}(\mathbb{R}^{d+1}_+)$ and satisfies
 \begin{equation}\label{limf}
   \lim_{(\gamma,\delta)\to (0,\infty)}\|f_{\gamma,\delta}-f\|_{\alpha,2}=0.
 \end{equation}
\end{thm}
 In the order to prove this theorem we need the following Lemmas.
\begin{lem}\label{lem1}
  Let $\varphi$ and $\psi$ be two-Weinstein wavelets  satisfying the assumptions ($A_1$) and ($A_2$) and $f$ in $L^2_{\alpha}(\mathbb{R}^{d+1}_+)$. Then we have the following assertions:\\
  (i) The functions $(\check{f}*\overline{\varphi_a}\check{)}$ and $(\check{f}*\overline{\varphi_a}\check{)}*\psi_a$ belongs to $L^2_{\alpha}(\mathbb{R}^{d+1}_+)$ and we have
        \begin{equation}\label{fourier**}
          \mathcal{F}_{W}((\check{f}*\overline{\varphi_a}\check{)}*\psi_a)(\xi)= \mathcal{F}_{W}(f)(\xi) \overline{\mathcal{F}_{W}(\varphi_a)(\xi)}\mathcal{F}_{W}(\psi_a)(\xi),\quad \xi\in\mathbb{R}^{d+1}_+.
         \end{equation}
    (ii) We have the following inequality
     \begin{equation}\label{norm2fourier**}
          \|\mathcal{F}_{W}(\check{f}*\overline{\varphi_a}\check{)}*\psi_a\|_{\alpha,2}\leq \|f\|_{\alpha,2} \|\mathcal{F}_{W}(\varphi)\|_{\alpha,\infty} \|\mathcal{F}_{W}(\psi)\|_{\alpha,\infty}.
         \end{equation}

\end{lem}
\begin{proof}
  (i) According to the relations (\ref{fourierbar}), (\ref{fourierbar1}) and Proposition \ref{propconvol} ($v$) we obtain
  \begin{eqnarray*}
    \mathcal{F}_{W}((\check{f}*\overline{\varphi_a}\check{)})(\xi) &=&  \mathcal{F}_{W}(\check{f}*\overline{\varphi_a})(-\xi) \\
     &=& \mathcal{F}_{W}(\check{f})(-\xi)\mathcal{F}_{W}(\overline{\varphi_a})(-\xi)\\
     &=& \mathcal{F}_{W}(f)(\xi)\overline{\mathcal{F}_{W}(\check{\varphi}_a)(-\xi)}.
  \end{eqnarray*}
  Hence, we have
  \begin{equation}\label{f*ubar}
     \mathcal{F}_{W}((\check{f}*\overline{\varphi_a}\check{)})(\xi) = \mathcal{F}_{W}(f)(\xi)\overline{\mathcal{F}_{W}(\varphi_a)(\xi)}.
  \end{equation}
  After that, we consider the  function $G$ defined on $\mathbb{R}^{d+1}_+$ by
  \begin{equation*}
    G(\xi)= \mathcal{F}_{W}((\check{f}*\overline{\varphi_a}\check{)})(\xi).
  \end{equation*}
  Therefore
   \begin{equation*}
     \mathcal{F}_{W}((\check{f}*\overline{\varphi_a}\check{)}*\psi_a)(\xi)=\mathcal{F}_{W}(G*\psi_a)(\xi),\quad \xi\in\mathbb{R}^{d+1}_+.
   \end{equation*}
   According to Proposition \ref{propconvol} ($v$), we deduce that the function $G$ belongs to $L^2_{\alpha}(\mathbb{R}^{d+1}_+)$ and we have
   \begin{equation}\label{G*va}
      \mathcal{F}_{W}(G*\psi_a)(\xi)= \mathcal{F}_{W}(G)(\xi) \mathcal{F}_{W}(\psi_a)(\xi),\quad \xi\in\mathbb{R}^{d+1}_+.
   \end{equation}
   Finally, we obtain the result by combining the relations  (\ref{f*ubar}) and (\ref{G*va}).\\
  (ii) From the assertion (i), we have
  \begin{equation*}\label{fourier**}
          \int_{\mathbb{R}^{d+1}_+}|\mathcal{F}_{W}((\check{f}*\overline{\varphi_a}\check{)}*\psi_a)(\xi)|^2d\mu_\alpha(x)= \int_{\mathbb{R}^{d+1}_+}|\mathcal{F}_{W}(f)(\xi)|^2 |\mathcal{F}_{W}(\varphi_a)(\xi)|^2 |\mathcal{F}_{W}(\psi_a)(\xi)|^2d\mu_\alpha(x).
         \end{equation*}
  Therefore, according to Plancherel formula for Weinstein transform (\ref{Plancherel formula})  and the fact that  $\mathcal{F}_{W}(\varphi_a)$ and $\mathcal{F}_{W}(\psi_a)$ belongs to $L^\infty_{\alpha}(\mathbb{R}^{d+1}_+)$, we deduce that
  \begin{equation*}
     \|\mathcal{F}_{W}(\check{f}*\overline{\varphi_a}\check{)}*\psi_a\|_{\alpha,2}\leq \|f\|_{\alpha,2} \|\mathcal{F}_{W}(\varphi_a)\|_{\alpha,\infty} \|\mathcal{F}_{W}(\psi_a)\|_{\alpha,\infty}.
  \end{equation*}
  In the end, we conclude the result from the relation (\ref{fourierfia}).
\end{proof}

\begin{lem}
  Let $\varphi$ and $\psi$ be two-Weinstein wavelets  satisfying the assumptions ($A_1$) and ($A_2$) . Then the function  $ K_{\gamma,\delta}$ defined as follows
  \begin{equation}\label{Kgamadelta}
   K_{\gamma,\delta}(\xi)=\frac{1}{C_{\varphi,\psi}}\int_{\gamma}^{\delta}\overline{\mathcal{F}_{W}(\varphi_a)(\xi)}\mathcal{F}_{W}(\psi_a)
   (\xi)\frac{da}{a}, \quad \xi\in \mathbb{R}^{d+1}_+
 \end{equation}
 satisfies, for almost all $\xi\in \mathbb{R}^{d+1}_+$:
 \begin{equation}\label{ineqKgamadelta}
   0< K_{\gamma,\delta}(\xi)\leq \frac{\sqrt{C_{\varphi}C_{\psi}}}{C_{\varphi,\psi}},
 \end{equation}
 and
 \begin{equation}\label{limKgamadelta}
    \lim_{(\gamma,\delta)\to (0,\infty)}K_{\gamma,\delta}(\xi)=1.
 \end{equation}
\end{lem}
\begin{proof}
  According to the Cauchy-Schwarz inequality and the identity (\ref{deftwowaveu=v}), we have for almost all $\xi\in \mathbb{R}^{d+1}_+$
  \begin{eqnarray*}
    |K_{\gamma,\delta}(\xi)| &\leq& \frac{1}{|C_{\varphi,\psi}|}\left(\int_{\gamma}^{\delta}|\mathcal{F}_{W}(\varphi_a)(\xi)|^2\frac{da}{a}\right)^\frac{1}{2} \left(\int_{\gamma}^{\delta}|\mathcal{F}_{W}(\psi_a)(\xi)|^2\frac{da}{a}\right)^\frac{1}{2}\\
     &\leq& \frac{1}{|C_{\varphi,\psi}|}\left(\int_{0}^{\infty}|\mathcal{F}_{W}(\varphi_a)(\xi)|^2\frac{da}{a}\right)^\frac{1}{2} \left(\int_{0}^{\infty}|\mathcal{F}_{W}(\psi_a)(\xi)|^2\frac{da}{a}\right)^\frac{1}{2}\\
     &=& \frac{\sqrt{C_{\varphi}C_{\psi}}}{C_{\varphi,\psi}}.
  \end{eqnarray*}
  On the other hand, it's clear that for almost all $\xi\in \mathbb{R}^{d+1}_+$
  $$\lim_{(\gamma,\delta)\to (0,\infty)}K_{\gamma,\delta}(\xi)=1.$$
  This completes the proof.
\end{proof}
\begin{thm}\label{thm2}
   Let $\varphi$ and $\psi$ be two-Weinstein wavelets  satisfying the assumptions ($A_1$) and ($A_2$) . Then the function  $f_{\gamma,\delta}$ defined by the relation (\ref{fgamadelta}) belongs to $L^2_{\alpha}(\mathbb{R}^{d+1}_+)$ and
satisfies
\begin{equation}\label{Fourfgamadelta}
  \mathcal{F}_{W}(f_{\gamma,\delta})(\xi)=\mathcal{F}_{W}(f)(\xi)K_{\gamma,\delta}(\xi),\quad \xi\in \mathbb{R}^{d+1}_+,
\end{equation}
where $K_{\gamma,\delta}$ is the function given by the relation (\ref{Kgamadelta}).
\end{thm}
\begin{proof}
  In a first step, we try to show that $f_{\gamma,\delta}$  belongs to $L^2_{\alpha}(\mathbb{R}^{d+1}_+)$. From, the definition of the  Weinstein continuous wavelet transform (\ref{contwave}) and the relations (\ref{symtrictrans}) and (\ref{deffiax}) we get
  \begin{equation*}
   f_{\gamma,\delta}(y)=\frac{1}{C_{\varphi,\psi}}\int_{\gamma}^{\delta}\int_{\mathbb{R}^{d+1}_+}
   (\check{f}*\overline{\varphi_a})(x) \tau^\alpha_y \psi_{a}(x)d\mu_\alpha(x)\frac{da}{a}.
  \end{equation*}
On the other hand, from the definition of the convolution product associated to the Weinstein transform given by the relation (\ref{defconvolution}), we have
\begin{eqnarray*}
  \int_{\mathbb{R}^{d+1}_+}
   (\check{f}*\overline{\varphi_a})(x) \tau^\alpha_y \psi_{a}(x)d\mu_\alpha(x) &=& \int_{\mathbb{R}^{d+1}_+}
   (\check{f}*\overline{\varphi_a}\check{)}(x) \tau^\alpha_y \psi_{a}(-x)d\mu_\alpha(x) \\
   &=& (\check{f}*\overline{\varphi_a}\check{)}*\psi_a(y).
\end{eqnarray*}
So, we have
\begin{equation}\label{fgamadeltaconvo}
   f_{\gamma,\delta}(y)=\frac{1}{C_{\varphi,\psi}}\int_{\gamma}^{\delta}(\check{f}*\overline{\varphi_a}\check{)}*\psi_a(y)\frac{da}{a}.
\end{equation}
By using Hölder's inequality for the measure  $\frac{da}{a}$, we obtain
\begin{equation*}
   |f_{\gamma,\delta}(y)|^2=\frac{1}{|C_{\varphi,\psi}|^2}\left(\int_{\gamma}^{\delta}\frac{da}{a}\right)\left(\int_{\gamma}^{\delta}
   \left|(\check{f}*\overline{\varphi_a}\check{)}*\psi_a(y)\right|^2\frac{da}{a}\right).
\end{equation*}
Aapplying Fubuni-Tonelli's theorem, after integrating the pervious inequality, we get
\begin{equation*}
  \int_{\mathbb{R}^{d+1}_+} |f_{\gamma,\delta}(y)|^2d\mu_\alpha(y)\leq \frac{1}{|C_{\varphi,\psi}|^2}\left(\int_{\gamma}^{\delta}\frac{da}{a}\right) \int_{\gamma}^{\delta}\left(\int_{\mathbb{R}^{d+1}_+}
   \left|(\check{f}*\overline{\varphi_a}\check{)}*\psi_a(y)\right|^2d\mu_\alpha(y)\right)\frac{da}{a}.
\end{equation*}
According to the Parseval's formula for the Weinstein transform (\ref{MM}) and the assertion $(i)$ of the Lemma \ref{lem1}, we deduce that
\begin{eqnarray*}
    \int_{\mathbb{R}^{d+1}_+} |f_{\gamma,\delta}(y)|^2d\mu_\alpha(y) &\leq &  \frac{1}{|C_{\varphi,\psi}|^2}\left(\int_{\gamma}^{\delta}\frac{da}{a}\right)\int_{\mathbb{R}^{d+1}_+} |\mathcal{F}_{W}(f)(\xi)|^2 \\
   & \times & \left(\int_{\gamma}^{\delta}|\mathcal{F}_{W}(\varphi_a)(\xi)|^2 |\mathcal{F}_{W}(\psi_a)(\xi)|^2\frac{da}{a}\right)d\mu_\alpha(\xi).
\end{eqnarray*}
On the other hand, from the identity (\ref{deftwowaveu=v}) and the  relation (\ref{fourierfia}), we have
\begin{equation*}
  \int_{\gamma}^{\delta}|\mathcal{F}_{W}(\varphi_a)(\xi)|^2 |\mathcal{F}_{W}(\psi_a)(\xi)|^2\frac{da}{a}\leq C_{\psi}\|\mathcal{F}_{W}(\varphi)\|_{\alpha, \infty}.
\end{equation*}
Therefore,
\begin{equation*}
   \int_{\mathbb{R}^{d+1}_+} |f_{\gamma,\delta}(y)|^2d\mu_\alpha(y) \leq  \frac{C_{\psi}}{|C_{\varphi,\psi}|^2} \left(\int_{\gamma}^{\delta}\frac{da}{a}\right)\|\mathcal{F}_{W}(\varphi)\|_{\alpha, \infty}\|\mathcal{F}_{W}(f)\|_{\alpha,2}.
\end{equation*}
Finally, from the Plancherel's formula for the Weinstein transform (\ref{Plancherel formula}), we get
\begin{equation*}
   \int_{\mathbb{R}^{d+1}_+} |f_{\gamma,\delta}(y)|^2d\mu_\alpha(y) \leq  \frac{C_{\psi}}{|C_{\varphi,\psi}|^2} \left(\int_{\gamma}^{\delta}\frac{da}{a}\right)\|\mathcal{F}_{W}(\varphi)\|_{\alpha, \infty}\|f\|_{\alpha,2}\leq \infty,
\end{equation*}
which implies that  $f_{\gamma,\delta}$  belongs to $L^2_{\alpha}(\mathbb{R}^{d+1}_+)$.

In the second step, We prove the relation (\ref{Fourfgamadelta}). Let $h\in\mathcal{S}_*(\mathbb{R}^{d+1})$, then $\mathcal{F}_{W}^{-1}(h)$ belongs to $mathcal{S}_*(\mathbb{R}^{d+1})$ and from the relation (\ref{fgamadeltaconvo}), we have
\begin{eqnarray}\label{fgamadeltaconvointegF}
  &&\!\!\!\!\!\!\!\!\!\!\!\!\!\!\!\! \int_{\mathbb{R}^{d+1}_+}  f_{\gamma,\delta}(y) \overline{\mathcal{F}_{W}^{-1}(h)(y)}d\mu_\alpha(y)  \nonumber\\ &=& \int_{\mathbb{R}^{d+1}_+} \left(\frac{1}{C_{\varphi,\psi}}\int_{\gamma}^{\delta}(\check{f}*\overline{\varphi_a}\check{)}*\psi_a(y)\frac{da}{a}\right) \overline{\mathcal{F}_{W}^{-1}(h)(y)}d\mu_\alpha(y).
\end{eqnarray}
Consider that
\begin{eqnarray*}
  &&\!\!\!\!\!\!\!\!\!\!\!\!\!\!\!\! \frac{1}{|C_{\varphi,\psi}|} \int_{\mathbb{R}^{d+1}_+} \int_{\gamma}^{\delta}\left|(\check{f}*\overline{\varphi_a}\check{)}*\psi_a(y) \overline{\mathcal{F}_{W}^{-1}(h)(y)}\right|\frac{da}{a}d\mu_\alpha(y)  \\ &=& \frac{1}{|C_{\varphi,\psi}|} \int_{\gamma}^{\delta}\left(\int_{\mathbb{R}^{d+1}_+} \left|(\check{f}*\overline{\varphi_a}\check{)}*\psi_a(y) \overline{\mathcal{F}_{W}^{-1}(h)(y)}d\mu_\alpha(y)\right)\right|\frac{da}{a}.
\end{eqnarray*}
By applying H\"{o}lder's inequality to right hand side of the previous equality, we obtain
\begin{eqnarray*}
  &&\!\!\!\!\!\!\!\!\!\!\!\!\!\!\!\! \frac{1}{|C_{\varphi,\psi}|} \int_{\gamma}^{\delta}\left(\int_{\mathbb{R}^{d+1}_+} \left|(\check{f}*\overline{\varphi_a}\check{)}*\psi_a(y) \overline{\mathcal{F}_{W}^{-1}(h)(y)}d\mu_\alpha(y)\right)\right|\frac{da}{a}  \\ &\leq & \frac{1}{|C_{\varphi,\psi}|} \int_{\gamma}^{\delta}\left(\int_{\mathbb{R}^{d+1}_+} \left|(\check{f}*\overline{\varphi_a}\check{)}*\psi_a(y) \overline{\mathcal{F}_{W}^{-1}(h)(y)}d\mu_\alpha(y)\right)\right|\frac{da}{a}.
\end{eqnarray*}
From the assertion $(ii)$ of Lemma \ref{lem1} and Plancherel's formula for Weinstein transform (\ref{Plancherel formula}), we get
\begin{eqnarray*}
  &&\!\!\!\!\!\!\!\!\!\!\!\!\!\!\!\! \frac{1}{|C_{\varphi,\psi}|} \int_{\gamma}^{\delta}\left(\int_{\mathbb{R}^{d+1}_+} \left|(\check{f}*\overline{\varphi_a}\check{)}*\psi_a(y) \overline{\mathcal{F}_{W}^{-1}(h)(y)}d\mu_\alpha(y)\right)\right|\frac{da}{a}  \\ &\leq & \frac{1}{|C_{\varphi,\psi}|} \left(\int_{\gamma}^{\delta}\frac{da}{a}\right)\|\mathcal{F}_{W}(\varphi)\|_{\alpha, \infty} \|\mathcal{F}_{W}(\psi)\|_{\alpha, \infty} \|h\|_{\alpha, 2}\|f\|_{\alpha, 2}.
\end{eqnarray*}
Then, from Fubini's theorem, the right hand side of the relation (\ref{fgamadeltaconvointegF}) can also be
written in the form
\begin{equation*}
 \frac{1}{C_{\varphi,\psi}}\int_{\gamma}^{\delta} \left(\int_{\mathbb{R}^{d+1}_+} (\check{f}*\overline{\varphi_a}\check{)}*\psi_a(y) \overline{\mathcal{F}_{W}^{-1}(h)(y)}d\mu_\alpha(y)\right)\frac{da}{a}.
\end{equation*}
Moreover, according to the Parseval's formula for the Weinstein transform (\ref{MM}) and the assertion $(i)$ of the Lemma \ref{lem1}, the pervious integral becomes
\begin{equation*}
 \frac{1}{C_{\varphi,\psi}}\int_{\gamma}^{\delta} \left(\int_{\mathbb{R}^{d+1}_+} \mathcal{F}_{W}(f)(\xi)\overline{\mathcal{F}_{W}(\varphi_a)(\xi) } \mathcal{F}_{W}(\psi_a)(\xi) \overline{h(\xi)}d\mu_\alpha(\xi)\right)\frac{da}{a}.
\end{equation*}
By applying Fubini's theorem to this last integral, we have
\begin{eqnarray}\label{inteFK}
  &&\!\!\!\!\!\!\!\!\!\!\!\!\!\!\!\! \int_{\mathbb{R}^{d+1}_+} \mathcal{F}_{W}(f)(\xi)  \left(\frac{1}{C_{\varphi,\psi}}\int_{\gamma}^{\delta}\overline{\mathcal{F}_{W}(\varphi_a)(\xi) } \mathcal{F}_{W}(\psi_a)(\xi) \frac{da}{a}\right)\overline{h(\xi)}d\mu_\alpha(\xi)  \nonumber\\ &=& \int_{\mathbb{R}^{d+1}_+} \mathcal{F}_{W}(f)(\xi)  K_{\gamma,\delta}(\xi)\overline{h(\xi)}d\mu_\alpha(\xi).
\end{eqnarray}
On the other hand, by applying the Parseval's formula for the Weinstein transform (\ref{MM}) to the left hand side of the relation (\ref{fgamadeltaconvointegF}), it takes the  integral form
\begin{equation}\label{Ffdeltagamapsi}
  \int_{\mathbb{R}^{d+1}_+} \mathcal{F}_{W}(f_{\gamma,\delta})(\xi)\overline{h(\xi)}d\mu_\alpha(\xi).
\end{equation}
Finally, from the  relations (\ref{inteFK}) and (\ref{Ffdeltagamapsi}), we obtain for all $h$ in $\mathcal{S}_*(\mathbb{R}^{d+1})$
\begin{equation*}
  \int_{\mathbb{R}^{d+1}_+} \left(\mathcal{F}_{W}(f_{\gamma,\delta})(\xi)-\mathcal{F}_{W}(f)(\xi)  K_{\gamma,\delta}(\xi)\right)\overline{h(\xi)}d\mu_\alpha(\xi)=0,
\end{equation*}
and we deduce that
\begin{equation*}
   \mathcal{F}_{W}(f_{\gamma,\delta})(\xi)=\mathcal{F}_{W}(f)(\xi)K_{\gamma,\delta}(\xi),\quad \xi\in \mathbb{R}^{d+1}_+.
\end{equation*}
\end{proof}
We now return to the proof of the Theorem \ref{mainthm}.
\begin{proof} \emph{of Theorem \ref{mainthm}}. At first, from Theorem \ref{thm2}, the function $f_{\gamma,\delta}$ belongs to $L^2_{\alpha}(\mathbb{R}^{d+1}_+)$. Next, according to Plancherel's formula for Weinstein transform (\ref{Plancherel formula}) and Theorem \ref{thm2}, we get
\begin{eqnarray*}
  \|f_{\gamma,\delta}-f\|_{\alpha,2} &=&   \int_{\mathbb{R}^{d+1}_+} |\mathcal{F}_{W}(f_{\gamma,\delta}-f)(\xi)|^2\alpha(\xi) \\
   &=&  \int_{\mathbb{R}^{d+1}_+} |\mathcal{F}_{W}(f)(K_{\gamma,\delta}(\xi)-1)|^2\alpha(\xi) \\
   &=& \int_{\mathbb{R}^{d+1}_+} |\mathcal{F}_{W}(f)|^2|(1-K_{\gamma,\delta}(\xi))|^2\alpha(\xi).
\end{eqnarray*}
Another time, according to Theorem \ref{thm2} we have for almost  all $\xi \in \mathbb{R}^{d+1}_+$
$$\lim_{(\gamma,\delta)\to (0,\infty)}|\mathcal{F}_{W}(f)|^2|(1-K_{\gamma,\delta}(\xi))|^2=0,$$
and there exists  a positive constant $C$ such that
$$\lim_{(\gamma,\delta)\to (0,\infty)}|\mathcal{F}_{W}(f)|^2|(1-K_{\gamma,\delta}(\xi))|^2\leq C|\mathcal{F}_{W}(f)|^2,$$
with $|\mathcal{F}_{W}(f)|^2$ is in $L^1_{\alpha}(\mathbb{R}^{d+1}_+)$. Thus, we conclude the relation (\ref{limf}) from the dominated convergence theorem.
\end{proof}

\end{document}